\documentclass{tpms}
\newtheorem{teor}{Theorem}
\newtheorem{cor}{Corollary}
\newtheorem{prop}{Proposition}

\newtheorem{lem}{Lemma}

\theoremstyle{definition}

\usepackage{color}
\usepackage{amssymb}

\numberwithin{equation}{section}

\begin{document}

\title{A note on last-success-problem}

\author{J.M. Grau Ribas}
\address{Departamento de Matemáticas, Universidad de Oviedo\\
Avda. Calvo Sotelo s/n, 33007 Oviedo, Spain}
\email{grau@uniovi.es}


\begin{abstract}
We consider the Last-Success-Problem with $n$ independent Bernoulli random variables with parameters $p_i>0$. We improve the lower bound provided by F.T. Bruss for the probability of winning and provide an alternative proof to the one given in \cite {BR2}
for the lower bound  ($1/e$) when $R:=\sum_{i=1}^n (p_i/(1-p_i))\geq1$. We also consider a modification of the game which consists in not considering it a failure when all the random variables take the value of 0 and the game is repeated as many times as necessary until a  $"1"$  appears.  { We prove that the probability of winning in this game when $R\leq1$ is lower-bounded by $0.5819...=\frac{1}{e-1} $}. Finally, we consider the variant in which the player can choose between participating in the game in its standard version or predict that all the random variables will take the value 0.
\end{abstract}
\maketitle

\subjclassname{ 60G40, 62L15}
\section{introduccion}
The Last-Success-Problem is the problem of maximizing the probability of
stopping on the last success in a finite sequence of Bernoulli trials. The
framework is as follows. There are $n$ Bernoulli random variables which are
observed sequentially. The problem is to find a stopping rule to maximize
the probability of stopping at the last "1". We restrict ourselves here to
the case in which the random variables are independent. This problem has
been studied by Hill and Krengel \cite{1992}, Hsiau and Yang \cite{2000} and
was simply and elegantly solved by F.T. Bruss in \cite{BR1} with the
following famous result.  Other recent papers related to this problem are \cite{fer16}, \cite{mono}, \cite{juego} and \cite{extension}.
\begin{teor}
(Odds-Theorem, F.T. Bruss 2000). Let $I_{1},I_{2},...,I_{n}$ be n
independent Bernoulli random variables with known $n$. We denote by ($%
i=1,...,n$) $p_{i}$, the parameter of $I_{i}$; i.e. ($p_{i}=P(I_{i}=1)$).
Let $q_{i}=1-p_{i}$ and $r_{i}=p_{i}/q_{i}$. We define the index
\begin{equation}
\mathbf{s}=
\begin{cases}
\max\{1\leq k\leq n: \sum_{j=k}^n r_j \geq 1\}, & \text{if $\sum_{i=1}^n
r_i\geq 1$}; \\
1, & \text{ otherwise.}%
\end{cases}%
\end{equation}
To maximize the probability of stopping on the last $"1"$ of the sequence,
it is optimal to stop on the first $"1"$ we encounter among the variables $%
I_{\mathbf{s}},I_{\mathbf{s}+1},...,I_{n}$.
The optimal win probability is given by
\begin{equation}
\mathcal{V}(p_1,...,p_n):=
\begin{cases}
 {\ \left( {\displaystyle{ \prod_{j=\mathbf{s} }^{n}q_j}} \right) } {%
\left( \displaystyle{\sum_{i=\mathbf{s} }^{n} r_i} \right)} , & \text{if $p_\mathbf{s}<1$};%
\vspace{2ex} \\
  {\ \displaystyle{\prod_{j=\mathbf{s+1} }^{n}q_j}}, & \text{if $p_\mathbf{s}=1.$
}%
\end{cases}%
\end{equation}
\end{teor}

Henceforth, we will denote by $\mathcal{G}(p_{1},...,p_{n})$ the game
consisting of pointing to the last 1 of the sequence $\{I_{1},...,I_{n}\}$,
where $0<p_{i}=P(I_{i}=1)$ for all $i=1,...,n$. We denote $%
R_{k}:=\sum_{i=k}^{n}\frac{p_{i}}{1-p_{i}}$ and $Q_{k}:=%
\prod_{i=k}^{n}(1-p_{i})$. The index $\mathbf{s}$ in Theorem 1 will be
called the \emph{optimal threshold} and the probability of winning, using
the optimal strategy, will be denoted by $\mathcal{V}(p_{1},...,p_{n})$.

Bruss also presented in \cite{BR1} the following bounds for the probability
of winning.

\begin{teor}
Let $\mathbf{s}$ be the optimal threshold for the game $\mathcal{G}%
(p_1,...,p_n)$, then
\begin{equation*}
    \mathcal{V}(p_1,...,p_n)> R_\mathbf{s} e^{-R_\mathbf{s}}.
\end{equation*}
\end{teor}

He subsequently presented an addendum  \cite{BR2} with the following result
for the case in which $R_{1}\geq 1$.
\begin{teor} \label{bruscota} If $  R_1\geq 1$ then
\begin{equation}\label{cota} \mathcal{V}(p_1,...,p_n)>\frac{1}{e}.
\end{equation}
\end{teor}

Very recently this result has been improved in \cite{otro} as follows.

\begin{teor} \label{ot} If $  R_1\geq 1$ then
$$ \mathcal{V}(p_1,...,p_n)\geq \left(1-\frac{1}{n+1}\right)^n>\frac{1}{e}.$$
\end{teor}

  This bound is, in fact, the same one that is proposed as Exercise      18 in \cite{fer} (Chapter 5,
Section 5) :$$ \mathcal{V}(p_1,...,p_n)\geq \left(1+\frac{1}{n}\right)^{-n}>\frac{1}{e}.$$ 
 However, Theorem 4 does not improve (\ref{cota}) significantly for large $n$ .

In the present paper, sharper  lower bounds are established for the
probability of winning than those presented above. In passing, we provide a
very different proof of Theorem \ref{bruscota} from that of Bruss.

In those cases where $ p_ {i} <1 $ for all $ i $, if all the random variables are zero, then the player fails. This suggests a variant (Variant I) of the standard game in which this is not considered a failure and the game is repeated as many times as necessary until a 1 appears. We study this variant in Section 3, where we will see that the typical value of $1/e$ for the
lower bound of the probability of winning is replaced by $\frac{1}{e-1}%
=0.5819...$.

We also consider the possibility that the player can choose between participating
in the game in its standard version or predict that all the random variables
will have a value of 0. The study of this variant (Variant II) is very
straightforward, but it is pleasing to discover that $1/e$ is the lower bound
for the probability of winning in all cases.

The final section summarizes the results obtained with respect to the lower
bounds for the probability of winning and establishes that the game with the
greatest probability of winning is Variant I.

\section{Lower bound for the case in which $R_{\mathbf{s}}=\infty $}

\begin{teor}
\label{inf}If $\mathbf{s}$ is the optimal threshold and $R_\mathbf{s}=\infty$, then
\begin{equation*}
\mathcal{V}(p_1,...,p_n)\geq {\left( \frac{  n - \mathbf{s}} {  n - \mathbf{s}+ R_{\mathbf{s}+1} }
\right) }^ {  n - \mathbf{s}}>\frac{1}{e^{R_{\mathbf{s}+1}}}>\frac{1}{e}.
\end{equation*}
\end{teor}

\begin{proof}
If $\mathbf{s}$ is the optimal threshold and $R_{\mathbf{s}}=\infty $, this means that $p_{\mathbf{s}}=1
$ and $R_{\mathbf{s}+1}<1.$ In this case, the probability of winning is
\begin{equation*}
\prod_{i=\mathbf{s}+1}^{n}(1-p_{i}).
\end{equation*}%
Minimizing $\prod_{i=\mathbf{s}+1}^{n}(1-x_{i})$ with respect to $x_{i}$ subject to
the constraint
\begin{equation*}
\sum_{i=\mathbf{s}+1}^{n}x_{i}/(1-x_{i})=R_{\mathbf{s}+1}
\end{equation*}%
shows (using Lagrange multiplier technique) that this minimum is obtained by
\begin{equation*}
 {x_{\mathbf{s}+1}}=....=x_{n}=\frac{R_{\mathbf{s}+1}}{R_{\mathbf{s}+1}+n-\mathbf{s} }
\end{equation*}
and its value
\begin{equation*}
{\left( \frac{n-\mathbf{s}}{n-\mathbf{s}+R_{\mathbf{s}+1}}\right) }^{n-\mathbf{s}}.
\end{equation*}
This is decreasing with $n$ always above its limit, which is $e^{-R_{\mathbf{s}+1}}>e^{-1}$.
\end{proof}

\section{Lower bound for the case in which $1\leq R_{\mathbf{s}}\leq\infty $}

The proof presented here  is very different from the Bruss's proof  and is based on the construction of a
problem with a lower probability of winning (always $>1/e$), adding a
sufficiently large number of Bernoulli random variables with
the same parameter. Previously, however, let us see several preparatory
lemmata.

\begin{lem}
\label{menor1} If $p\in (0,1)$ and $\mathbf{x}\geq 1$, then
\begin{equation*}
(1-p)\frac{\frac{p}{1-p}+\mathbf{x}}{\mathbf{x}}\leq 1.
\end{equation*}
\end{lem}

\begin{proof}
 {It should be borne in mind that $$(1-p)\frac{\frac{p}{1-p}+\mathbf{x}}{\mathbf{x}}=p \left(\frac{1}{\mathbf{x}}-1\right)+1.$$}
\end{proof}

\begin{lem}
\label{menor2} If $\{p,P\}\subset (0,1)$ with $p\leq P$ and $%
1\leq \mathbf{x}\leq \frac{1}{1-   {P}}$, then
\begin{equation*}
\frac{(1-p)}{1-P}\frac{\frac{p}{1-p}+\mathbf{x}-\frac{P}{1- P}}{\mathbf{x}}\leq 1.
\end{equation*}
\end{lem}

\begin{proof}  {
$$\frac{(1-p)}{1-P}\frac{\frac{p}{1-p}+\mathbf{x}-\frac{P}{1- P}}{\mathbf{x}}-1= \frac{(p-P) ((P-1) \mathbf{x}+1)}{(P-1)^2 \mathbf{x}} $$
 $$  \frac{(1-p)}{1-P}\frac{\frac{p}{1-p}+\mathbf{x}-\frac{P}{1- P}}{\mathbf{x}}\leq 1 \Longleftrightarrow  \frac{(p-P) ((P-1) \mathbf{x}+1)}{(P-1)^2 \mathbf{x}}\leq 0.$$} Now, the last inequality is true,
 seeing as $$ p\leq P \textrm{ and } 1\leq \mathbf{x}\leq \frac{1}{1-P}.$$
\end{proof}

 {
\begin{lem}\label{infinito}
  If $\{p,\mathbf{x}\}\subset (0,1)$, then $$(1-p)\left( \frac{p}{1-p}+\mathbf{x}\right)\leq 1.$$
\end{lem}
\begin{proof} $$(1-p)\left( \frac{p}{1-p}+\mathbf{x}\right)=p+\mathbf{x}-p \mathbf{x}.$$
If we now assume that $p+\mathbf{x}-p \mathbf{x}=1+\epsilon>1$, then we have the following contradiction
$$p=\frac{1+\epsilon-\mathbf{x}}{1-\mathbf{x}}>1.$$
\end{proof}}
   {
We will denote by  $ \lceil  \alpha \rceil$ the least integer greater than or equal to $\alpha$.}

\begin{lem} \label{bueno}   {  Let $\mathfrak{f}:[0,1] \longrightarrow
\mathbb{R}$ be the following function}
  {
$$
\mathfrak{f}(x):=
\begin{cases}
x\cdot(\lceil 1/x\rceil-1)\cdot(1-x)^{\lceil 1/x\rceil-2}, & \text{if
$0<x<1$}; \\
1/e, &  \text{if $x=0$};\\ 1, &  \text{if $x=1$. }
\end{cases}%
$$}
  {Then, we have that $ \mathfrak {f} $ is continuous and strictly
increasing in $ [0,1] $ and, consequently, that $ 1 / e <\mathfrak {f}
(x) $ for all $ x \in (0,1] $.}
\end{lem}
\begin{proof}   {The function $ \mathfrak {f} $ can be rewritten as follows}

  {
$$
\mathfrak{f}(x)=
\begin{cases}
x\cdot n\cdot(1-x)^{n-1}, & \text{if $ \frac{1}{n+1}\leq x<\frac{1}{n}$
and $n \in \mathbb{N}$}; \\
1/e, &  \text{if $x=0$}; \\ 1, &  \text{if $x=1$. }
\end{cases}%
$$}

  {
so it is clear that the function $ \mathfrak{f}\in
C^1\left(\frac{1}{n+1}, \frac{1}{n}\right)$. }
  {
Since $\mathfrak{f}(x)=x$ for all $x \in [1/2,1)$, it follows that
$$ \lim_{x\rightarrow 1^{-}}\mathfrak{f} \left( x\right)=1.$$}
  {Moreover,
$$\lim_{x\rightarrow \frac{1}{n}^{-}}\mathfrak{f} \left(
x\right)=\left(1-\frac{1}{n}\right)^{n-1}=\frac{1}{n} (n-1)
\left(1-\frac{1}{n}\right)^{n-2}=\lim_{x\rightarrow
\frac{1}{n}^{+}}\mathfrak{f} ( x )$$
so it follows that $ \mathfrak{f} $ is also continuous at $ \frac{1}{n}
$ for all $ n \in \mathbb{N} $.}
  {
On the other hand, for every $x\in \left( \frac{1}{n+1},
\frac{1}{n}\right)$ we have that $$\mathfrak{f}'(x)= -n (1-x)^{n-2} (n
x-1)>0.$$ Thus, $\mathfrak{f}$ is strictly increasing in $(0,1]$ and it
only remains to prove that $\mathfrak{f}$ is continuous at $0$.
To do so, note that}
  {$$x\in \left( \frac{1}{n+1}, \frac{1}{n}\right) \Longrightarrow
\left(\frac{n}{n+1}\right)^n=  \mathfrak{f}\left(\frac{1}{n+1}\right) <
\mathfrak{f}(x)<\mathfrak{f}\left(\frac{1}{n}\right)
=\left(1-\frac{1}{n}\right)^{n-1},  $$
so, taking limits:
$$\frac{1}{e} =\lim_{n\rightarrow\infty}
\left(\frac{n}{n+1}\right)^n\leq
\lim_{x\rightarrow0^+} \mathfrak{f}(x)\leq\lim_{n\rightarrow\infty}
\left(1-\frac{1}{n}\right)^{n-1}
=\frac{1}{e}\Longrightarrow\lim_{x\rightarrow0^+}
\mathfrak{f}(x)=\frac{1}{e}.$$
and the proof is complete.}
\end{proof}

\begin{lem}
\label{emenuevo}  { Let  the game $\mathcal{G}(p_{1},...,p_{n})$ have
 $$   {p_{i}=p<1}\textrm{  for } n-\lceil 1/p \rceil+2\leq i\leq n \textrm{ and } 1\leq n-\lceil 1/p \rceil+2.  $$
Hence, the optimal threshold is $\mathbf{s}=n-\lceil 1/p \rceil+2$ and the probability of winning is
\begin{equation*}
\mathcal{V}(p_{1},...,p_{n})=
p \cdot(\lceil 1/p\rceil-1)\cdot(1-p)^{\lceil 1/p\rceil-2}
 =  {\mathfrak{f}(p)>1/e.}
\end{equation*}}

\end{lem}
\begin{proof} {If $n-\lceil 1/p \rceil+2\geq 1$, then $\mathbf{s}=n-\lceil 1/p\rceil +2\geq 1$ is
the optimal threshold, given that
\begin{equation*}
R_{\mathbf{s}}=(n-\mathbf{s}+1)\frac{p}{1-p}=(\lceil 1/p\rceil -1)\frac{p}{1-p}\geq 1
\end{equation*}
\begin{equation*}
R_{\mathbf{s}+1}=(n-\mathbf{s})\frac{p}{1-p}=(\lceil 1/p\rceil -2)\frac{p}{1-p}<1
\end{equation*}
On the other hand,
$$Q_{\mathbf{s}}=(1-p)^{n-\mathbf{s}+1}=(1-p)^{\lceil 1/p \rceil-1}.$$
And  finally, using Lemma \ref{bueno}, we have

$$\mathcal{V}(p_{1},...,p_{n})=R_{\mathbf{s}}Q_{\mathbf{s}}=
p \cdot(\lceil 1/p\rceil-1)\cdot(1-p)^{\lceil 1/p\rceil-2} =  {\mathfrak{f}(p)>\frac{1}{e}.} $$

}
\end{proof}

 {In the particular case that $ p $ is the inverse of a natural number, we have the following corollary with a   very sharp bound.}

\begin{cor}
\label{eme} Let $1<m\in \mathbb{N}$ and the game $\mathcal{G}%
(p_{1},...,p_{n})$ with $n-m+2\geq 1$ and $p_{i}=1/m$ for $n-m+2\leq i\leq n$.
Hence, the optimal threshold is $s=n-m+2$ and the probability of winning is
\begin{equation*}
\mathcal{V}(p_{1},...,p_{n})={\left( \frac{-1+m}{m}\right) }^{-1+m}>1/e.
\end{equation*}
\end{cor}

\begin{lem}
\label{anadir} Let $\mathbf{s}$ be the optimal threshold and $\vartheta $
the probability of winning for the game $\mathcal{G}(p_{1},...,p_{n})$.  {Let $%
\mathbf{p}\leq\min \{p_{i}:i=\mathbf{s},...,n\}$}. Let us now consider the auxiliary
game $\mathcal{G}(p_{1},...,p_{n},p_{n+1})$, with $p_{n+1}=\mathbf{p}$, and let us denote by $\vartheta ^{\ast }$
the probability of winning for this game. Then:
\begin{equation*}
\vartheta \geq \vartheta ^{\ast }.
\end{equation*}
\end{lem}

\begin{proof}%

We denote by $V(t)$ the probability of winning when $t$  {is the  threshold used}  and by $V^{\ast }(t)$ the same probability for the auxiliary problem.   {Let us denominate by $\mathbf{s}^{\ast }$ the optimal threshold for the
problem $\mathcal{G}(p_{1},...,p_{n},p_{n+1})$.}

  {
We will assume, first, that $R_1\geq1$.}

 Given that $\mathbf{s}$ is the optimal threshold for the problem $\mathcal{G}%
(p_{1},...,p_{n})$, we have that:

$\bullet$ If $p_\mathbf{s}<1:$

\begin{equation*}
\mathbf{s}=\max\{k: \sum_{j=k }^n r_j\geq 1\} \text{ and } V(\mathbf{s}%
):=\left(\prod_{j=\mathbf{s} }^{n} q_j \right) \left(\sum_{i=\mathbf{s}
}^{n} r_i \right)=\vartheta.
\end{equation*}
Let us denominate by $\mathbf{s}^{\ast }$ the optimal threshold for the
problem $\mathcal{G}(p_{1},...,p_{n},p_{n+1})$. Thus,
\begin{equation*}
\mathbf{s^{\ast }}=\max \{k:r_{n+1}+\sum_{j=k}^{n}r_{j}\geq 1\}\text{ and }%
V^{\ast }(\mathbf{s}^{\ast }):=\left( q_{n+1}\prod_{j=\mathbf{s^{\ast }}%
}^{n}q_{j}\right) \left( r_{n+1}+\sum_{i=\mathbf{s^{\ast }}}^{n}r_{i}\right)
=\vartheta ^{\ast }.
\end{equation*}
 {
We will see what  $\mathbf{s}^{\ast }\in \{\mathbf{s},\mathbf{s} +1\}$. It's clear that $\mathbf{s}\leq \mathbf{s}^*$. Let us assume $\mathbf{s}+k=\mathbf{s}^*>\mathbf{s}+1$. Now, considering   that $r_{n+1}=\frac{\mathbf{p}}{1-\mathbf{p}}\leq r_i$, for all $i=\mathbf{s},...,n$, it  concluded that
$$ \sum_{i=\mathbf{s}+1}^nr_i\geq\sum_{i=\mathbf{s}+2}^nr_i+ \frac{\mathbf{p}}{1-\mathbf{p}} \geq \sum_{i=\mathbf{s}+k}^nr_i+ \frac{\mathbf{p}}{1-\mathbf{p}} \geq1$$
 which contradicts the fact that $\mathbf{s}$ is optimal. Thus,   $\mathbf{s}^{\ast }\in \{\mathbf{s},\mathbf{s} +1\}$, and  it suffices to prove that}
\begin{equation*}
V^{\ast }(\mathbf{s})\leq V(\mathbf{s})\text{ and }V^{\ast }(\mathbf{s}%
+1)\leq V(\mathbf{s}).
\end{equation*}
Now, making $\mathbf{x}:=\sum_{i=\mathbf{s}}^{n}r_{i}\geq1$, from Lemma \ref{menor1}, we
have that
\begin{equation*}
\frac{V^*(\mathbf{s})}{V(\mathbf{s})}=q_{n+1}\cdot\frac{r_{n+1}+ \mathbf{x}}{\mathbf{x}}\leq 1.
\end{equation*}
On the other hand, since $ \mathbf{x}-r_\mathbf{s} <1 $ and therefore $ \mathbf{x}<\frac{1}{1-p_\mathbf{s}}$, from Lemma \ref{menor2}, we have that

\begin{equation*}
\frac{V^*(\mathbf{s}+1)}{V(\mathbf{s})}=\frac{q_{n+1}}{q_\mathbf{s}}\cdot%
\frac{r_{n+1}+ \mathbf{x}- r_\mathbf{s}}{\mathbf{x}}\leq 1.
\end{equation*}
 {
$\bullet$ If $p_\mathbf{s}=1:$
\begin{equation*}
  \sum_{j=\mathbf{s}+1 }^n r_j< 1  \text{ and } V(\mathbf{s}%
):= \prod_{j=\mathbf{s}+1 }^{n} q_j   =\vartheta.
\end{equation*}
Thus, reasoning as above, we have that $\mathbf{s}^*\in \{\mathbf{s},\mathbf{s}+1\}$ and then we  prove that \begin{equation*}
V^{\ast }(\mathbf{s})\leq V(\mathbf{s})\text{ and }V^{\ast }(\mathbf{s}%
+1)\leq V(\mathbf{s}).
\end{equation*}
$$V^*(\mathbf{s})=q_{n+1}\left(\prod_{j=\mathbf{s} +1 }^{n} q_j \right)=(1-p)\left(\prod_{j=\mathbf{s}+1  }^{n} q_j \right)=(1-p)V(\mathbf{s})<V(\mathbf{s})$$
$$V^*(\mathbf{s}+1)=\left( q_{n+1}\prod_{j=\mathbf{s}+1}^{n}q_{j}\right) \left( r_{n+1}+\sum_{i=\mathbf{s}+1}^{n}r_{i}\right)={(1-p)V(\mathbf{s})\left( \frac{p}{1-p}+\sum_{i=\mathbf{s}+1}^{n}r_{i}\right)}$$
Now, making $\mathbf{x}:=\sum_{i=\mathbf{s}+1}^{n}r_{i}<1$, from Lemma \ref{infinito}, we
have that $V(\mathbf{s})\geq V^*(\mathbf{s}+1)$.}

  {Finally, if $ R_1 <1 $, then $ \mathbf {s} = 1 $ and, reasoning as in the case $ p_s <1 $, we also have that $ \vartheta \geq \vartheta^{\ast} $. }

\end{proof}

\begin{teor}  \label{mas1}Let us consider the game $\mathcal{G}(p_{1},...,p_{n})$ and let $%
\mathbf{s}$ be the optimal threshold with {$1\leq R_{\mathbf{s}}\leq\infty $.  Let $p:=\min\{p_{i}:\mathbf{s}\leq i\leq n\}$, then
\begin{equation*}
\mathcal{V}(p_{1},...,p_{n}) >1/e.
\end{equation*}}
\end{teor}
\begin{proof}
Using Lemma \ref{anadir} repeatedly, we can build a sequence of games with a
non-increasing probability of winning by attaching successive independent
Bernoulli random variables with parameter $ p $.
When the attachment process has been carried out as many times as is
necessary,  we shall be able to use Lemma \ref{emenuevo} and shall have
\begin{equation*}
\mathcal{V}(p_{1},...,p_{n})\geq \mathcal{V}(p_{1},...,p_{n},p,...,p)\geq
p \cdot(\lceil 1/p\rceil-1)\cdot(1-p)^{\lceil 1/p\rceil-2}=  {\mathfrak{f}(p)>\frac{1}{e}}.
\end{equation*}
\end{proof}
This result improves the lower bound, $1/e$, quite ostensibly when all the
parameters $p_{i}$ are moderately far from 0.
\subsection{ANNEX: The case $p_{i}=p$ for all $i$}
We end this section with a result   for the
case in which all the Bernoulli random variables have the same parameter $p$. This was
treated, with a certain degree of imprecision, in \cite{mal}.
\begin{prop}
Let us consider the game $\mathcal{G}(p_{1},...,p_{n})$, with $p_{i}=p<1$.
Thus,

$\bullet $ If $n\geq (\lceil 1/p\rceil -1)$, then $\mathbf{s}=n-\lceil 1/p\rceil +2$
is the optimal threshold and
\begin{equation*}
\mathcal{V}(p,...,p)= p \cdot(\lceil 1/p\rceil-1)\cdot(1-p)^{\lceil 1/p\rceil-2}.
\end{equation*}
$\bullet $ If $n<(\lceil 1/p\rceil -1)$, then $\mathbf{s}=1$ is the optimal threshold
and
\begin{equation*}
\mathcal{V}(p,...,p)=n\cdot p \cdot(1-p)^{n-1}.
\end{equation*}
\end{prop}
\begin{proof}
 {$\bullet $  If $n\geq (\lceil 1/p\rceil -1)$, then  the conditions of the Lemma \ref{emenuevo} are met.}

 {
$\bullet $  If $n<(\lceil 1/p\rceil -1)$, then $R_{1}=n\frac{p}{1-p}<1$ and hence $\mathbf{s}=1$
is the optimal threshold and, moreover, $Q_{1}= (1-p)^n$. Thus
$$ \mathcal{V}(p,...,p)=R_1 Q_1=n\cdot p \cdot(1-p)^{n-1} .$$}
\end{proof}
\cite{mal}  addressed  the problem differently, concluding that $1/e$ is
a lower bound for the probability of winning. However, the
optimal threshold that is considered in the aforementioned paper
\begin{equation*}
s^{\ast }=\left\lfloor n+1+\frac{1}{\log (1-p)}+\frac{1}{2}\right\rfloor
\end{equation*}%
is not correct, as $s^{\ast }$ does not always coincide with $s=n-\lceil
1/p\rceil +2$, obtained in the previous proposition. Although it must be
said that it is a very good estimate.

\section{Lower bound for the case in which $1>R_{1}$}

\begin{teor}
\label{menos1}If $R_{1}<1$, then
\begin{equation*}
\mathcal{V}(p_1,...,p_n)>R_1\,{\left( \frac{n}{n + R_1} \right) }^ n> R_1
e^{-R_1}.
\end{equation*}
\end{teor}

\begin{proof}
\begin{equation*}
\mathcal{V}(p_1,...,p_n)=R_1Q_1 =\left(\sum_{i=1}^n \frac{p_i}{1-p_i}%
\right)\prod_{i= 1}^n (1-p_i)
\end{equation*}
Considering $f(x_{1},...,x_{n})=\sum_{i=1}^{n}\frac{x_{i}}{1-x_{i}}
\prod_{i=1}^{n}(1-x_{i})$ and using Lagrange multiplier technique we have that the minimum value of $f$ with the constraint $%
R_{1}=\sum_{i=1}^{n}\frac{x_{i}}{1-x_{i}}$ is reached in $x_{1}=...=x_{n}=%
\frac{R_{1}}{R_{1}+n}$ and the minimum value of $f$ is
\begin{equation*}
R_{1}\,{\left( \frac{n}{n+R_{1}}\right) }^{n}>R_{1}e^{-R_{1}}.
\end{equation*}
\end{proof}

\section{Variant I: If there have been no $1^{\prime }s$, the game is
repeated}

In those cases in which $p_{i}<1$ for all $i$ the player may fail because he
has no chance to point to any\emph{\ "last 1"}, as all the variables are 0.
This suggests a variant of the original game in which the game is repeated
as many times as necessary until a 1 appears. Of course, if $p_{i}=1$ for
some $i$, then it will never be necessary to repeat the game.

\begin{prop}
If $\mathbf{s}$ is the optimal threshold for the game $\mathcal{G}%
(p_{1},...,p_{n})$, with $p_{i}<1$ for all $i$, then the probability of
winning with the new rule is

\begin{equation} \label{repetir}
\mathcal{V}^*(p_1,...,p_n)=\frac{\left(\sum_{i=\mathbf{s}}^n\frac{p_i}{1-p_i}%
\right)\prod_{i=\mathbf{s}}^n (1-p_i)}{1-\prod_{i=1}^n (1-p_i)}.
\end{equation}
\end{prop}

\begin{proof}
Obviously, the optimal strategy, with this rule, is the same as in the game
in its original version. The difference lies only in the probability of
winning, which is conditioned by $\sum_{i=1}^{n}I_{i}>0$. Thus, bearing in
mind that
\begin{equation*}
P\left( \sum_{i=1}^{n}I_{i}>0\right) =1-P\left(
I_{1}=I_{2}=...=I_{n}=0\right) =1-\prod_{i=1}^{n}(1-p_{i}),
\end{equation*}
we have that
\begin{equation*}
\mathcal{V}^{\ast }(p_{1},...,p_{n})=P\left( \mathsf{WIN}%
|\sum_{i=0}^{n}I_{i}>0\right) =\frac{\mathcal{V}(p_{1},...,p_{n})}{%
1-\prod_{i=1}^{n}(1-p_{i})}.
\end{equation*}
\end{proof}

The cases in which  {$1\geq R_{1}$ and $1< R_{1}$} require a different treatment.

\subsection{The case in which  {$1\geq R_{1}$} }

\begin{equation*}
\label{0}
\end{equation*}

\begin{teor}
\label{1} If  { $1\geq R_{1}$}   for the game $\mathcal{G}%
(p_1,...,p_n)$, then

\begin{equation*}
\mathcal{V}^*(p_1,...,p_n)>\frac{R_1}{-1 + e^{R_1}} { \geq } \frac{1}{-1 + e }
=0.5819...
\end{equation*}
\end{teor}

\begin{proof}  {Taking into account  {(\ref{repetir})}, let us now consider \begin{equation*}
\mathcal{V}^*(x_1,...,x_n)=\frac{\left(\sum_{i=1}^n\frac{x_i}{1-x_i}\right)\prod_{i=1}^n (1-x_i)}{%
1-\prod_{i=1}^n (1-x_i)}.
\end{equation*}} Minimizing $\mathcal{V}^*$ with respect to $x_{i}$ subject to the constraint
\begin{equation*}
\sum_{i=1}^{n}\frac{x_{i}}{1-x_{i}}=R_{1},
\end{equation*}
shows (using Lagrange multiplier technique) that this minimum is obtained by
\begin{equation*}
x_1=x_2=...=x_n=\frac{R_1}{n+R_1}.
\end{equation*}
The minimum value is
\begin{equation*}
\frac{R_1 {\left( \frac{n}{n + R_1} \right) }^n}{1 - {\left( \frac{n}{n + R_1%
} \right) }^n}>\frac{R_1}{-1 + e^{R_1}} {\geq} \frac{1}{-1 + e } =0.5819...
\end{equation*}
\end{proof}

\subsection{The case in which  {$1< R_{1}<\infty$ }}
 
\begin{teor} {
\label{2} If $1< R_{1}<\infty$ for the game $\mathcal{G}%
(p_1,...,p_n)$, then
\begin{equation*}
\mathcal{V}^*(p_1,...,p_n)  >\frac{e^{-1}}{1-\prod_{i=1}^{n}(1-p_{i})}.
\end{equation*}}
\end{teor}

\begin{proof}
{Simply consider in \eqref{repetir}
the bound \eqref{cota} established in Theorem \ref{bruscota}.}

\end{proof}
{Here is a case where, even though $R_1>1$, $\mathcal{V}^{\ast }$ is lower-bounded by $\frac{1}{e-1}$.}

\begin{teor} If $p_i=1/n$ for $i=1,...,n$, then
$$ \mathcal{V}^{\ast }(p_{1},...,p_{n})>\frac{1}{e-1}=0.5819... $$
\end{teor}
\begin{proof}   { Obviusly, for $n=1$, $\mathcal{V}^*(1)=1>\frac{1}{e-1}$.}

For $n\geq 2$, the optimal threshold is $ \mathbf{s} = 2 $, seeing as $$R_2=(n-1)\frac{1/n}{1-1/n}=1\textrm{ and }R_3=(n-2)\frac{1/n}{1-1/n}<1 $$ and hence
$$\mathcal{V}^*(p_1,...,p_n) =\frac{R_2Q_2}{
1-\prod_{i=1}^n (1-p_i)}=\frac{\prod_{i=2}^n (1-p_i)}{%
1-\prod_{i=1}^n (1-p_i)}=\frac{(1-1/n)^{n-1}}{1-(1-1/n)^n}>\frac{1}{e-1}.$$
\end{proof}

\section{Variant II: The player can predict that there will be no $1^{\prime}s$}
In the game in its original version, one of the ways the player loses is
that all the variables are zero. We speculate what will occur if, as an
initial possible move, we allow the player to predict that all the random
variables will have the value 0 (i.e. there will not be $1^{\prime }s$). The
answer is simple: the new game is equivalent to the original game adding a
first random variable, $I_{0}$ with parameter $p_{0}=1$. Effectively,
stopping at stage 0 in the standard game is equivalent to predicting that
there will be no $1^{\prime }s$.  {In this variant, if the player predicts that all the random
variables will have the value 0, the probability of winning is  $ \prod_{i=1}^{n}(1-p_{i})=Q_{1} $
and the probability of winning in the standard game, when the optimal threshold is $\mathbf{s}$, is $\mathcal{V}(p_{1},...,p_{n})=R_\mathbf{s}  Q_\mathbf{s}.$ The probability of winning, let us call it $\mathcal{V^{\ast\ast }}(p_{1},...,p_{n})$, is hence
$$\mathcal{V^{\ast\ast }}(p_{1},...,p_{n})=\max\{Q_1,Q_\mathbf{s}R_\mathbf{s}\}.$$
Thus, the optimal strategy for \textbf{Variant II} is as follows:}

 {\begin{enumerate}
  \item If $Q_{1}=Q_\mathbf{s}R_\mathbf{s} $,    {it makes no difference} whether the player predicts  that there will be no $%
1^{\prime }s$ or plays the standard game.
  \item If $Q_{1}<Q_\mathbf{s}R_\mathbf{s} $, then play the standard game.
  \item If $Q_{1}>Q_\mathbf{s}R_\mathbf{s} $, then predict that there will be no $1^{\prime }s$.
\end{enumerate}}

Having said so, we now have the following
result.

\begin{teor}
Taking $R_{1}=\sum_{i=1}^{n} \frac{p_{i}}{1-p_i} $, the optimal strategy for
\textbf{Variant II} is as follows:
 
\begin{enumerate}
  \item If $R_{1}=1$, it is indifferent to predict that there will be no $%
1^{\prime }s$ or play the standard game.

  \item If $R_{1}>1$, then play the standard game.

  \item If $R_{1}<1$, then predict that there will be no $1^{\prime }s$.
\end{enumerate}
In any case, the probability of winning, $\mathcal{V^{\ast
\ast }}(p_{1},...,p_{n})$, is greater than $1/e$.
\end{teor}

 \begin{proof} {
It suffices to bear in mind that the probability of winning when betting
that there will be no $1^{\prime }s$ is  $ \prod_{i=1}^{n}(1-p_{i})=Q_{1} $
and the probability of winning the standard game, when the optimal threshold is $\mathbf{s}$, is $\mathcal{V}(p_{1},...,p_{n})=R_\mathbf{s}  Q_\mathbf{s}.$}

(1) If $R_1=1$, the optimal threshold is $\mathbf{s}=1$ and we have
  $$R_{\mathbf{s}}Q_{\mathbf{s}}=R_1Q_1=Q_1.$$
(2) If $R_{1}>1$, the optimal threshold is $\mathbf{s}\geq 1$ with $R_{\mathbf{s}}\geq 1$ and we have
  {{ $$\textrm{ if }\mathbf{s}=1,\textrm{ then  } R_{\mathbf{s}}Q_{\mathbf{s}}=R_{1}Q_{1}> Q_{1}.$$}
 $$\textrm{ if } \mathbf{s} >1, R_{\mathbf{s}}Q_{\mathbf{s}}\geq  Q_{\mathbf{s}} > Q_{1}.$$}
(3) If $R_{1}<1$, the optimal threshold is $\mathbf{s}=1$ and we have
  $$ R_{\mathbf{s}}Q_{\mathbf{s}}= R_{1}Q_{1}<Q_{1} .$$
The probability of winning for this variant is greater than $1/e$ because,
as has been stated, it is actually equivalent to a standard game with $p_{0}=1$.
\end{proof}
In summary, if the probability of winning for the standard game is not
greater than $1/e$, then the probability that all the variables are zero is
greater than $1/e$. And vice versa: if the probability that all the
variables are zero is not greater than $1/e$, then the probability of
winning for the standard game is greater than $1/e$. Of course it can also
occur that the probability is greater than $1/e$ in both cases, but it can
never be the case that both probabilities are simultaneously less than $1/e$.

\section{Summary and conclusions}

We first present the results related to the bounds for the probability of
winning. Finally, we find that Variant I always has a higher probability of
winning than Variant II, except in the case of $p_{i}=1$ for some $i$, in
which case the three games are in fact equivalent.

\begin{teor}
Let $\mathbf{s}$ be the optimal threshold for the game $\mathcal{G}%
(p_1,...,p_n)$ and \newline  { $p:= \min\{p_i: i=\mathbf{s},...,n\}$,} then

 {
\begin{equation*}
\mathcal{V}(p_1,...,p_n)\geq
\begin{cases}
{{\left( 1 - p \right) }^ {-1 + \frac{1}{p}}}  {>\frac{1}{e}} , & \text{if $1\leq R_1 < \infty$ }; \\
R_1\,{\left( \frac{n }{n + R_1 } \right) }^{n }, & \text{if $R_1<1$}; \\\max\{e^{-R_{\mathbf{s}+1}} , {{\left( 1 - p \right) }^ {-1 + \frac{1}{p}}}\}, &   { \text{if $R_1=\infty $}}.
\end{cases}
\end{equation*}}

\end{teor}

\begin{proof}

For the case $1\leq R_{1}< \infty $, see Theorem \ref{mas1}.

For the case $ R_{1}<1$, see Theorem \ref{menos1}.

For the case   {$ R_1=R_{\mathbf{s}}=\infty$}, see Theorems \ref{mas1} and \ref{inf}.

\end{proof}

\begin{teor}
Let $\mathbf{s}$ be the optimal threshold for the game $\mathcal{G}%
(p_{1},...,p_{n})$, then
\begin{equation*}
\mathcal{V^{\ast }}(p_{1},...,p_{n})
\begin{cases}
\geq\frac{R_{1}{\left( \frac{n}{n+R_{1}}\right) }^{n}}{1-{\left( \frac{n}{n+R_{1}%
}\right) }^{n}}>\frac{R_{1}}{ e^{R_{1}}-1}\geq \frac{1}{e-1}, & \text{if $%
R_{1}\leq 1$ }; \\ \geq {\frac{e^{-1}}{1-\prod_{i=1}^{n}(1-p_{i})}}, & \text{if $1<R_{1}<\infty$}; \\
= \mathcal{V}(p_{1},...,p_{n}), &  \text{if $R_{1}=\infty $}.%
\end{cases}%
\end{equation*}
\end{teor}

\begin{proof}
For the case $R_{1}\leq 1$, see Theorem \ref{1}.

For the case $1< R_{1}<\infty $, see Theorem \ref{2}.

  {For the case $R_1=\infty $, $p_\mathbf{s}=1$, and, consequently,} $\mathcal{V^*}(p_1,...,p_n)=\mathcal{V}%
(p_1,...,p_n)$.
\end{proof}

\begin{teor}
If $p_{i}=1$ for some $i=1,...,n$,  then all three games are equivalent by construction.
Otherwise, 
 
\begin{equation*}
\mathcal{V}(p_{1},...,p_{n})\leq \mathcal{V^{\ast \ast }}(p_{1},...,p_{n})<\mathcal{V^{\ast }}(p_{1},...,p_{n}),
\end{equation*}
\end{teor}

\begin{proof}   {
If $R_{\textbf{s}}\geq 1,\textrm{ then } R_{1}\geq 1$ and, by Theorem 11, } we have $\mathcal{V}(p_{1},...,p_{n})=\mathcal{V^{\ast
\ast }}(p_{1},...,p_{n})$. Moreover,
\begin{equation*}
\mathcal{V^{\ast }}(p_{1},...,p_{n})=\frac{\mathcal{V}(p_{1},...,p_{n})}{%
1-\prod_{i=1}^{n}(1-p_{i})}>\mathcal{V}(p_{1},...,p_{n}).
\end{equation*}
If $R_{s}<1$, then $\mathcal{V}(p_{1},...,p_{n})<\mathcal{V^{\ast \ast }}%
(p_{1},...,p_{n})$. Moreover,   {in this case $R_1 < 1$, and so, by Theorem 13,}
\begin{equation*}
\mathcal{V^*}(p_1,...,p_n)> \frac{R_1}{-1 + e^{R_1} }\text{ and }\mathcal{%
V^{**}}(p_1,...,p_n) =\prod_{i=1}^n(1-p_i).
\end{equation*}
Maximizing $\prod_{i=1}^{n}(1-x_{i})$ with respect to $x_{i}$ subject to the
constraint $R_{1}=\sum_{i=1}^{n}\frac{x_{i}}{1-x_{i}}$ shows that this
maximum is reached when all the $x_{i}$ are 0 except for one,  which
takes the value  $\frac{R_{1}}{1+R_{1}}$. So that
\begin{equation*}
\mathcal{V^{\ast \ast }}(p_{1},...,p_{n})\leq 1-\frac{R_{1}}{1+R_{1}}=\frac{1%
}{1+R_{1}}<\frac{R_{1}}{-1+e^{R_{1}}}<\mathcal{V^{\ast }}(p_{1},...,p_{n}).
\end{equation*}
\end{proof}

\textbf{Acknowledgements} The author wish to thank the anonymous referees for their many insightful comments
and suggestions that helped to improve the paper.


\begin{thebibliography}{99}


 \bibitem{otro} Allaart, P.  and Islas J.A. (2016).  A sharp lower bound for choosing the maximum of an independent sequence. \emph{ J. Appl. Prob.} Volume 53 (4),   {1041--1051}.

\bibitem{BR1} Bruss, F.T. (2000). Sum the odds to one and stop.
{\em Ann. Probab.}  28,  no. 3, 1384--1391.

\bibitem{BR2} Bruss, F.T.  (2003). A note on bounds for the odds theorem of optimal
stopping.
{\em Ann. Probab. } 31,  no. 4, 1859--1861.

\bibitem{mono} { Bruss F.T. (2019). Odds-Theorem and Monotonicity.
 \emph{Math. Applicanda. } 47(1), pp. 25--43.}



\bibitem{fer16}  {Ferguson, T.S. (2016). The Sum-the-Odds Theorem with Application to a Stopping
Game of Sakaguchi, \emph{Math. Appl.}, 44 (1). 45–61.}

\bibitem{fer} Ferguson, T.S. (2006). Optimal stopping and applications. \newline
 Electronic Text at http://www.math.ucla.edu/{$\thicksim$}tom/Stopping/Contents.html

\bibitem{extension}  {Grau Ribas, J.M. (2020). An extension of the Last-Success-Problem. \emph{Stat. Probab. Lett}. Volume 156, Article 108591.}

\bibitem{juego}  {Grau Ribas, J.M. (2019). A Turn-Based Game Related to the Last-Success-Problem. \emph{Dyn. Games Appl}. 10.1007/s13235-019-00342-y}

\bibitem{1992} Hill T. P. and Krengel, U. (1992). A prophet inequality related to the secretary problem.\emph{ Contemp. Math}. 125 209--215.

\bibitem{2000} Hsiau, S. R. and Yang, J. R. (2000). A natural variation of the standard secretary problem. \emph{Statist. Sinica}. 10. 639-646.



\bibitem {mal} Kohn, W. (2014). Last Success Problem: Decision Rule and Application. Available at SSRN: https://ssrn.com/abstract=2441250 or http://dx.doi.org/10.2139/ssrn.2441250







\end{thebibliography}
\end{document}